\documentclass[preprint,11pt]{imsart}

\RequirePackage[OT1]{fontenc}
\RequirePackage{amsthm,amsmath}
\RequirePackage[numbers]{natbib}
\RequirePackage[colorlinks,citecolor=blue,urlcolor=blue]{hyperref}

\usepackage{amsmath,amsfonts,amsthm,amssymb,color}
\usepackage{fancyhdr}
\usepackage{subfig,balance,algorithm,algorithmic,lineno,multirow}

\usepackage{epsfig}
\usepackage{graphicx}
\usepackage{xcolor}
\usepackage{float}
\usepackage{verbatim}
\usepackage{bbm}

\newcommand{\bB}{\mathbb{B}}
\newcommand{\bR}{\mathbb{R}}
\newcommand{\bN}{\mathbb{N}}

\newcommand{\cF}{\mathcal{F}}
\newcommand{\cG}{\mathcal{G}}

\newcommand{\cN}{\mathcal{N}}

\newcommand{\cFt}{\mathcal{F}_{t}}

\newcommand{\Exp}{{\sf E}}
\newcommand{\Var}{{\sf Var}}

\newcommand{\Pro}{{\sf P}}

\newcommand{\bv}{\mathsf{b}}

\def\argmax{\mathop{\mbox{argmax}}}
\def\argmin{\mathop{\mbox{argmin}}}

\newtheorem{lemma}{Lemma}

\newtheorem{proposition}{Proposition}

\theoremstyle{remark} \theoremstyle{lemma}
\theoremstyle{definition} \theoremstyle{corol}
\theoremstyle{proposition} \theoremstyle{condition}

\newcommand{\ind}[1]{\mathbbm{1}_{\{#1\}}}

\startlocaldefs
\numberwithin{equation}{section}
\theoremstyle{plain}
\newtheorem{thm}{Theorem}[section]
\endlocaldefs

\begin{document}

\begin{frontmatter}
\title{Sequential Design for Computerized Adaptive Testing that Allows for Response Revision}
\runtitle{Design for CAT that allows for response revision}

\begin{aug}
\author{\fnms{Shiyu} \snm{Wang} \thanks{S.\ Wang is with the  Department of Statistics, University of Illinois, Urbana-Champaign, IL 61820, USA, e-mail: swang86@illinois.edu}} ,
\author{\fnms{Georgios} \snm{Fellouris} \thanks{G.\ Fellouris is with the  Department of Statistics, University of Illinois, Urbana-Champaign, IL, USA, e-mail: fellouri@illinois.edu}}
\and
\author{\fnms{Hua-Hua} \snm{Chang} \thanks{H-H.\ Chang is with the  Department of Psychology, University of Illinois, Urbana-Champaign, IL 61820, USA, e-mail: hhchang@illinois.edu
}}
\runauthor{ Wang, Fellouris and Chang}
\end{aug}

\begin{abstract}
In computerized adaptive testing (CAT), items (questions) are selected in real time based on the already observed responses, so that the ability of the examinee can be estimated as accurately as possible. This is typically formulated as a  non-linear, sequential, experimental design problem with binary observations that correspond to the true or false responses.  However,  most items in practice are multiple-choice and dichotomous models do not make full use of the available data. Moreover, CAT has been heavily criticized for not allowing test-takers to review and revise their answers.  In this work, we propose a novel CAT design that is based on the polytomous nominal response model and  in which test-takers are allowed to revise their responses  at any time during the test. We show that as the number of administered items goes to infinity, the proposed estimator is
  (i) strongly consistent for any item selection and revision strategy and  (ii)  asymptotically normal when the   items are selected to maximize the Fisher information at the  current ability  estimate and the number of revisions is smaller than the  number of items. We also present the findings of a simulation study that supports our asymptotic results.
\end{abstract}

\begin{keyword}[class=MSC]
\kwd[Primary ]{62L05}
\kwd[; secondary ]{62P15}
\end{keyword}

\begin{keyword}
\kwd{Computerized Adaptive Testing, Experimental Design,  Large Sample Theory, Item Response Theory, Martingale Limit Theory,  Nominal Response Model, Response Revision, Sequential Design}
\end{keyword}

\end{frontmatter}

\section{Introduction}
A main goal in educational assessment is the accurate estimation of each test-taker's ability, which is a kind of latent trait. In a conventional paper-pencil test, this estimation is based on the examinee's responses to a preassembled set of items. On the other hand,  in Computerized Adaptive Testing (CAT), items are selected in real time, i.e.,  the next item depends on the already observed responses.  In this way, it is possible to tailor the difficulty of the items to the examinee's ability and estimate the latter more efficiently than that in a paper-pencil test. This is especially true for examinees at the two extreme ends of the ability distribution,  who may  otherwise receive items either too difficult or too easy.   CAT was originally proposed by Lord \cite{r1} and with the  rapid development of modern technology it has become popular for many kinds of measurement tasks, such as educational testing, patient reported outcome, and quality of life measurement.  Examples of large-scale CATs include the Graduate Management Admission Test (GMAT), the National Council Licensure Examination (NCLEX) for nurses, and the Armed Services Vocational Aptitude Battery (ASVAB) \cite{rc}.

The two main tasks in a CAT, i.e., ability estimation and item selection,  depend heavily on Item Response Theory (IRT) for modeling the response of the examinee. This is done by specifying the  probability of a correct answer as a function of  certain item-specific parameters and the ability level, which is represented by a scalar parameter   $\theta$. For example, in the two-parameter logistic (2PL) model,  the probability of a correct answer is equal to $H(a(\theta-b))$, where $H(x)=e^x/ (1+e^x)$. The item parameters for this model are the difficulty parameter $b$ and the discrimination parameter $c$. The 2PL is an extension of the Rasch model \cite{r10}, which corresponds to the special case that $a=1$. On the other hand, the 2PL can be generalized by adding a parameter that captures  the probability of guessing the right answer (3PL model).

Given the IRT model, a standard approach for item selection, proposed by  Lord \cite{r9},  is to select the item that maximizes the Fisher information of the model at each step. For the above logistic models,  this item selection procedure suggests selecting the item with difficulty parameter $b$ equal to $\theta$.
Since $\theta$ is unknown, this implies that the difficulty parameter for  item $i$, $b_i$,  should be equal to $\theta_{i-1}$,  the estimate of $\theta$ based on the first $i-1$ observations. As it was suggested by Wu \cite{wu85,wu86}, the adaptive estimation of $\theta$ can be achieved via
 a likelihood-based approach, instead of  the non-parametric,  Robbins-Monro \cite{robmon} algorithm that had been originally proposed by Lord \cite{r8} and can be very inefficient with binary data \cite{lai}. When $\theta_i$ is selected to be the  Maximum Likelihood Estimator (MLE) of $\theta$ based on the first $i$ observations, the resulting final ability estimator was shown to be strongly consistent and asymptotically normal
 by Ying and Wu \cite{yingwu}  under the Rasch model and   Chang and Ying \cite{r2} under the 2PL and the 3PL models.

 However, while the design and analysis of CAT in the educational and statistical literature typically  assumes  dichotomous IRT models,  most operational CAT programs employ  multiple-choice items, for which  dichotomous models are unable to differentiate among the (more than one)  incorrect answers.
This implies  a   loss of efficiency that could be  avoided if a polytomous IRT model, such as Bock's  \textit{nominal response model} \cite{r1}, was used instead. Indeed, based on a simulation study,  de Ayala \cite{r3} found that a CAT based on the nominal response model leads to a more accurate ability estimator than a CAT that is based on the 3PL model. However, to our knowledge,   there has not been any theoretical support to this claim. In fact,  generalizing the results in \cite{r2} and \cite{yingwu}
in the case of the nominal response model  is a very non-trivial problem, since for items with $m\geq 2$ categories there are
 $2(m-1)$ parameters need to be selected at each step and there is no convenient, explicit form for the item parameters that maximize the Fisher information. 

Our first contribution is that we study theoretically the design of a CAT that is based on the  nominal response model with an arbitrary number of categories. Specifically, assuming that the response are conditionally independent given the selected items and that the item parameters belong to a  bounded set, we prove (Theorem \ref{theo1})  that the MLE of $\theta$ (with any item selection strategy) is strongly consistent as the number of administered items goes to infinity.   If additionally each item is selected  to maximize the Fisher information at the current MLE of the ability level, we show that the  MLE of $\theta$ becomes asymptotically normal and efficient (Theorem \ref{theo3}). The significance of our first work is the design of a CAT that is based on the polytomous nominal response model using the full capacity of multiple-choice items, in comparison to a dichotomous model that wastes information by treating them as binary(true/false).
 
 Our second main contribution in this work  is that we show that a CAT design with the nominal response model can be used to alleviate the major criticism that is addressed to CAT: the fact that test-takers are not allowed to review and revise their answers during the test. Indeed, it is commonly believed that  revision conflicts with the adaptive nature of CAT, and, hence, decreases efficiency and leads to biased ability estimation \cite{r10,r11,r12,r13}. Thus,  none of the currently operational CAT programs allows for response revision, which is allowed by the traditional paper-pencil tests. This has become a main concern for both examinees and testing companies, and for this reason some test programs have decided to switch from CAT to other modes of testing  \cite{r24}.

On the other hand, it is clear that the response revision feature can provide a more user-friendly  environment, by  helping alleviate the test-takers' anxiety. It may  even lead to a more reliable  ability estimation, by reducing the  measurement error that is associated with  careless mistakes (that the examinees may correct). Therefore, it has been a long-standing problem to incorporate the response revision feature  in CAT.  Certain modified  designs have been proposed for this purpose, such as  CAT with restricted review models  \cite{r10}  and multistage adaptive testing \cite{r24}, and it has been argued that if appropriate review and revision rules are set, there will be no  impact on the estimation accuracy and efficiency \cite{r26,r13}.  However, all these studies (that either support or oppose response revision in CAT)  rely on Monte Carlo simulation experiments and lack  a theoretical foundation.

In this work, we propose a CAT design that allows for response revision and we establish its asymptotic properties under a rigorous statistical framework. Specifically,  assuming that we have multiple-choice items with $m\geq 3$ categories, our main idea is to exploit the  flexibility  of the  nominal response model in order to obtain an algorithm that gives partial credit when the examinee corrects a previously wrong answer. Moreover,  our setup for revision is very   flexible: each  examinee is allowed to revise a previous answer at any time during the test as long as each item is revised at most  $m-2$ times. However, this leads to a non-standard  experimental design problem which differs from the traditional CAT setup in two ways. First, items need to be selected at certain random times, which are  determined by the examinee. Second,  information is now accumulated  at two time-scales: that of the observations/ responses and that of the items.

In order to address this problem, we  assume,  as in the context of the standard CAT, that responses  from different items are conditionally independent and that  the nominal response model  governs the first response to each item. However, we now further assume that whenever an item is revised during the test, the new response will follow the conditional pmf  of the nominal response model given that previous answers cannot be repeated. 
Our final ability estimator  is the maximizer of the  conditional likelihood  of all observations (first responses and  revisions) given the selected item parameters and the observed  decisions of the examinee to revise or not at each step.  We show   (Theorem \ref{LLNCAT}) that this estimator
 is strongly consistent for any item selection  and  revision strategy. When in particular the items are selected to maximize the Fisher information of the nominal response model  at the current ability estimate and, additionally, the  number of revisions is "small" relative to the number of items, we show that  the proposed estimator is also  asymptotically normal, with the same asymptotic variance as that in the regular CAT  (Theorem \ref{CLTRCAT}).

From a practical point of view, the most important feature of our approach is that it incorporates revision without the need to calibrate any additional item parameters than the ones used in a regular CAT that is based on the nominal response model.  Indeed, if a  dichotomous IRT model was employed instead,  incorporating revision would require calibrating  the probability of switching from a correct answer to a wrong answer and vice-versa for all items in the pool. This is a very difficult task in practice and  probably infeasible for large-scale implementation.

The rest of the  paper is organized as follows. In Section 2, we introduce the nominal response model and its main properties.  In Section 3, we focus on the  design and asymptotic  analysis  of a regular CAT that is based on the nominal response model. In Section 4, we formulate the problem of CAT  design that allows for response revision, we present the proposed scheme and establish its asymptotic properties.   In Section 5, we present the findings of a simulation study  that  illustrates our theoretical results. We conclude in Section 6.

\section{Nominal Response Model}
In this section, we introduce the nominal response model, which is the IRT model that we will use for the design of CAT in next sections. Throughout the paper, we focus on the case of a single examinee, whose ability is quantified by a scalar parameter  $\theta \in \bR$ that is the quantity of interest. Thus, the underlying probability measure  is denoted by  $\Pro_{\theta}$.

 Let $X$ be the response to a generic  multiple-choice item with  $m\geq 2 $ categories. That is,  $X= k$ when the examinee chooses category $k$,  where   $1 \leq k \leq m$, and the nominal response model assumes that
\begin{align} \label{nominal}
 \Pro_{\theta}(X=k)&= \frac{\exp(a_{k}\theta+c_{k})}{\sum_{h=1}^{m}\exp(a_{h}\theta+c_{h})},  \quad 1 \leq k \leq m,
\end{align}
where  $\{a_{k}, c_{k}\}_{1 \leq k \leq m}$ are real numbers that  satisfy
\begin{align} \label{trivial}
\sum_{k=1}^{m}|a_{k}| \neq 0 \quad  \text{and}  \quad   \sum_{k=1}^{m}|c_{k}| \neq 0
\end{align}
and the following identifiability conditions:
\begin{align} \label{identify}
\sum_{k=1}^{m}a_{k}=\sum_{k=1}^{m}c_{k}=0.
\end{align}
The latter assumption implies that one of the $a_{k}$'s and one of the $c_{k}$'s is completely determined by the others. As a result, without loss of generality we can say that the distribution of $X$ is completely determined by the ability parameter $\theta$ and the vector  $\bv:=(a_{2},\ldots, a_{m}, c_{2},\ldots, c_{m})$. In order to simplify the notation we will write:
\begin{align} \label{nominal2}
 p_{k}(\theta; \bv) &:= \Pro_{\theta}(X=k),  \quad 1 \leq k \leq m.
\end{align}
Note that in the case of binary data ($m=2$), the nominal response model recovers the 2PL model with  discrimination parameter $2|a_1|$ and difficulty parameter $-c_2/ a_2$. In particular,  \eqref{identify} implies $a_1=-a_2$, $c_1=-c_2$ so that
\begin{align*}
p_{2}(\theta;\bv)= 1- p_{1}(\theta;\bv) &= \frac{\exp(2a_{2}\theta+ 2c_{2})}{1+\exp(2a_2\theta+2c_2)}.
\end{align*}
The log-likelihood and the score function of $\theta$   take the form
\begin{align}
\ell(\theta;\mathbf{b}, X) &:= \log \Pro_{\theta}(X) = \sum_{k=1}^{m}  \log \bigl(p_{k}(\theta;\bv)  \bigr)  \,  \ind{X=k}  \label{loglikeli}  \\
s(\theta;\bv, X) &:= \frac{d}{d \theta} \ell(\theta;\bv,X) = \sum_{k=1}^{m}  \bigl[ a_{k}  - \bar{a}(\theta;\bv) \bigr] \, \ind{X=k},  \label{score}
\end{align}
 where $\bar{a}(\theta;\bv)$ is the following weighted average of the $a_k$'s:
\begin{align} \label{alphas}
 \bar{a}(\theta;\bv) &:= \sum_{h=1}^{m} a_{h} \, p_{h}(\theta;\bv).
\end{align}
The Fisher information of $X$ as a function of $\theta$ takes the form:
\begin{align} \label{fisher}
 J(\theta;\bv) &:= \Var_{\theta}[s(\theta;\bv, X)] = \sum_{k=1}^{m} \Bigl( a_{k}-\bar{a}(\theta;\bv) \Bigr)^{2} \,  p_{k}(\theta;\bv),
\end{align}
whereas  the derivative of  $s(\theta;\bv,X)$ with respect to $\theta$
does not depend on $X$ and is equal to  $- J(\theta;\bv)$,  which justifies the following notation:
\begin{equation} \label{derivative}
 s'(\tilde{\theta};\bv) := \frac{d}{d\theta} s(\theta;\bv,X)\Big|_{\theta=\tilde{\theta}} = - J(\tilde{\theta};\bv) .
\end{equation}
Moreover,  $J(\theta;\bv)$ is positive and has an upper bound that is independent of $\theta$, in particular,
\begin{align} \label{fisher2}
0<  J(\theta;\bv) \leq \sum_{k=1}^{m}  a_{k}  ^{2} \,  p_{k}(\theta;\bv) \leq  \sum_{k=1}^{m} a_{k} ^{2} \leq m \, a^{*}(\bv) ,
\end{align}
where we denote $a^{*}(\bv)$ and $a_{*}(\bv)$  as the maximum and minimum of the $a_k$'s respectively, i.e..
$$
a^{*}(\bv)  := \max_{1 \leq k \leq m} a_{k} \quad \text{and} \quad a_{*}(\bv) := \min_{1 \leq k \leq m} a_{k}.
$$
  The first inequality holds in \eqref{fisher2} because the $a_{k}$'s cannot be identical, due to  \eqref{trivial}-\eqref{identify}. However,
 while for   any given $\theta \in \bR$ and $\bv$ we have $a_{*}(\bv) < \bar{a}(\theta;\bv) < a^{*}(\bv)$,  from \eqref{nominal} it follows that $\bar{a}(\theta;\bv) \rightarrow  a_{*}(\bv)$ as $\theta\rightarrow -\infty$ and    $\bar{a}(\theta;\bv) = a^{*}(\bv)$ as $\theta\rightarrow +\infty$ and, consequently,
\begin{align} \label{limiting2}
\lim_{|\theta|\rightarrow\infty}J(\theta;\bv) &=0,
\end{align}
i.e.  the  Fisher information of the model goes to 0 as the ability level goes to $\pm\infty$ for any given item parameter.

Since in practice items are drawn from a given item bank,  we will assume that  $\bv$ takes value in a compact subset  $\bB$ of $\bR^{2m-2}$, a  rather realistic assumption whenever we have a given item bank.
This assumption will be technically useful  through the following result (Maximum Theorem), whose proof can be found, for example, in \cite{econ}, p. 239.

\begin{lemma} \label{lem1}
If  $g: \bR \times \bB \rightarrow \bR$ is a continuous function, then
 $  \sup_{\bv \in \bB} \,  g(\cdot,  \bv)$ and $\inf_{\bv \in \bB} \,  g(\cdot , \bv)$ are also continuous functions. Thus, if  $x_n \rightarrow x_0$, then
  $\sup_{\bv \in \bB} |g(x_n, \bv)-g(x_{0}, \bv)|  \rightarrow 0$.
\end{lemma}

As a first illustration of this result, note that since  $ J(\theta;\bv)$ is jointly continuous,  then
\begin{equation} \label{bound}
\theta \rightarrow  J_{*}(\theta):=\inf_{\bv \in \bB} J(\theta;\bv)  \quad
 \text{and} \quad \theta \rightarrow  J^{*}(\theta) :=  \sup_{\bv \in \bB} J(\theta;\bv)
\end{equation}
are also continuous functions. Moreover, from  Lemma \ref{lem1} and \eqref{fisher2}  it follows that there is a universal  in $\theta$ upper (but not lower) bound on the Fisher information that corresponds to each ability level, i.e.,
\begin{equation} \label{K}
0 < J_{*}(\theta) \leq  J^{*}(\theta) \leq K:= m \, \sup_{\bv \in \bB} \left( a^{*}(\bv) \right)^{2},  \quad \forall \, \theta \in \bR.
\end{equation}

\section{Design of standard CAT with  Nominal Response Model}

\subsection{Problem formulation}
In this section  we focus on the design of a  CAT with a fixed number of items, $n$,  each of which has $m \geq 2$ categories. Let $X_{i}$ denote the response to item $i$, thus,   $X_{i}= k$ if the examinee chooses category $k$ in item $i$, where  $1 \leq k \leq m$ and $1 \leq i \leq n$.  We assume that the responses are governed by the nominal response model,  defined in \eqref{nominal}, so that
\begin{align} \label{model}
\Pro_{\theta}(X_{i}=k) &:= p_{k}(\theta; \bv_i), \quad 1 \leq k \leq m, \; 1 \leq i \leq n,
\end{align}
where $\theta$ is the scalar parameter of interest that represents the ability of the examinee and  $\bv_i:=(a_{i2},\ldots, a_{im}, c_{i2},\ldots, c_{im})$ is a $\bB$-valued  vector  that characterizes item $i$ and satisfies \eqref{trivial}-\eqref{identify}. Moreover, we assume that  the responses  are conditionally independent given the selected items, in the sense that
\begin{align} \label{independence}
\Pro_{\theta}  (X_{1}, \ldots, X_{i} \, | \,  \bv_1,\ldots,\bv_i )
&= \prod_{j=1}^{i} \Pro_{\theta}  (X_{j} |  \bv_{j} ), \quad 1 \leq i \leq n.
\end{align}
However, while in a conventional paper-pencil test  the parameters $\bv_1, \ldots, \bv_n$ are deterministic,
in a CAT they are random , determined in real time based on the already observed responses. Specifically,  let $\cF^{X}_{i}$ be the information contained in the first $i$ responses, i.e., $\cF_{i}^{X}:=\sigma(X_1,....,X_i)$. Then, each  $\bv_{i+1}$ is an $\cF^{X}_{i}$-measurable, $\bB$-valued  random vector and, as a result, despite assumption \eqref{independence},  the responses are far from  independent and, in fact, they may have a complex dependence structure.

The problem in CAT is to find an \textit{ability estimator}, $\hat{\theta}_{n}$, at the end of the test, i.e.,  an $\cF^{X}_{n}$-measurable estimator of $\theta$,   and  \textit{an item selection strategy},  $(\bv_{i+1})_{1 \leq i \leq n-1}$, so that the accuracy of $\hat{\theta}_{n}$ be optimized.  If we  were able to select each item $i$ so that   $J(\theta;\bv_i)= J^{*}(\theta)$,  where $J^{*}(\theta)$ is the maximum  Fisher information an item can achieve (recall \eqref{bound}) at the true ability level $\theta$,  then we could use  standard asymptotic theory in order to obtain  an  estimator, $\hat{\theta}_{n}$,  such as the MLE, that is asymptotically efficient,  in the sense that
$
\sqrt{n}( \hat{\theta}_n -\theta) \rightarrow \cN\left (0, [J^{*}(\theta)]^{-1} \right)
$ as $n \rightarrow \infty$.  Of course, this is not a feasible item selection strategy, as it requires knowledge of $\theta$, the parameter we are trying to estimate! Nevertheless, we can make use of the adaptive nature of CAT and select items that  maximize the Fisher information at the current estimate of the ability level.  That is,  $\bv_{i+1}$ can be chosen to belong to
\begin{equation} \label{item_select}
\argmax_{\bv\in \bB}J( \hat{\theta}_{i};\bv),
\end{equation}
where  $\hat{\theta}_{i}$ is an estimate of the ability level that is based on the first $i$ responses,  $1 \leq i \leq n$.

We should note that  this item selection method assumes that  each $\bv_i$  can take any value in $\bB$. Of course, this is not the case in practice, where a given item bank has a finite number of items and there are restrictions on  the exposure rate of the items \cite{rr}.
 Nevertheless, this  item selection strategy will provide a benchmark for the best possible performance  that can be expected, at least in an asymptotic sense.

\subsection{Adaptive Maximum Likelihood Estimation of $\theta$}

The item selection strategy \eqref{item_select} calls  for an adaptive estimation of the examinee's ability during the test process. From the
 conditional independence assumption \eqref{independence}  it follows that the conditional log-likelihood function of the first $i$ responses given the selected items takes the form
\begin{align*}
L_{i}(\theta) &:= \log \Pro_{\theta}  (X_{1}, \ldots, X_{i} |  \bv_1,\ldots,\bv_i ) = \sum_{j=1}^{i} \ell(\theta;\bv_j, X_j),
\end{align*}
where $\ell(\theta;\bv_j, X_j)$ is the log-likelihood that corresponds to the $j^{th}$ response and is determined by the nominal response model, according to \eqref{loglikeli}.  Then, the  corresponding   score function takes the form
\begin{align}\label{scoren}
  S_i(\theta)& :=\frac{d}{d\theta} L_{i}(\theta) =\sum_{j=1}^{i} s(\theta;\bv_j, X_j),
 \end{align}
 where  $s(\theta;\bv_j, X_j)$ is  the score function that corresponds to the $j^{th}$ item and is  defined according to \eqref{score}.  We would like our estimate for $\theta$ after the first $i$ observations to be  the root of $S_{i}(\theta)$. Unfortunately, this root does not  exist  for every $1 \leq i \leq n$. Indeed,
$S_{i}(\theta)$ does not have a root   when all acquired responses either correspond to the category with  the largest   $a$-value,  or to the category with smallest  $a$-value. In other words, the root of $S_{i}(\theta)$ exists and is unique for every $i>n_{0}$, where
 \begin{align*}
n_{0}:=\max \Big\{i \in \{1, \ldots, n\}:  & X_{j} =  \argmax\{a_{jk}\}_{k=1}^{m} \; \forall j \leq i \\
 \text{or} \;  &  X_{j} \in \argmin\{a_{jk}\}_{k=1}^{m}  \; \forall j \leq i
\Big\},
\end{align*}
For example,  in a CAT with $n=7$ items of $m=4$ categories where for each item the largest (resp. smallest) $a$-value is associated with category $4$(resp. $1$), for the sequence of responses $1,1,1,3,4,1,3$  we have $n_{0}=3$.

For $i \leq n_{0}$, an initial estimation procedure is needed to estimate the ability parameter. A possible initialization strategy is to set  $\hat{\theta}_0=0$ and, for every $i \leq n_{0}$,
$\hat{\theta}_i=\hat{\theta}_{i-1}+ d$ (resp. $\hat{\theta}_i=\hat{\theta}_{i-1}-  d$) if the acquired responses have  the largest (resp. smallest) $a$-value, whereas $d$ is a predetermined constant.  


\subsection{Asymptotic Analysis}

 We now focus on the asymptotic properties of $\hat{\theta}_n$ as $n \rightarrow \infty$, thus we will assume  without loss of generality that
 $\hat{\theta}_{n}$ is the root of $S_n(\theta)$ for sufficient large values of $n$. Specifically, we will establish the strong consistency of  $\hat{\theta}_{n}$ for any item selection strategy and its asymptotic normality and efficiency when the information maximizing item selection \eqref{item_select} is adopted. Both properties rely heavily on the martingale property of the score function, $S_n(\theta)$, which is established in the following proposition.

\begin{proposition} \label{prop1}
For any item selection strategy, $(\bv_n)_{n \in \bN}$,  the score process $\{S_n(\theta)\}_{n \in \bN}$   is a $(\Pro_\theta, \{\cF_{n}\}_{n \in \bN})$-martingale with bounded increments, mean $0$ and predictable variation  $\langle S(\theta) \rangle_{n} =I_n(\theta)$, where
\begin{align}\label{fishern}
I_n(\theta) &:= \sum_{i=1}^{n} J(\theta;\bv_i), \quad n \in \bN,
\end{align}
and $J(\theta;\bv_i)$ is the Fisher information of the $i^{th}$ item, defined in \eqref{fisher}. Moreover,  for any $\tilde{\theta}$ we have
\begin{align} \label{deriv2}
S^{'}_n(\tilde{\theta}) &:= \frac{d}{d\theta} S(\theta) \Big|_{\theta=\tilde{\theta}}=-I_n(\tilde{\theta}).
\end{align}
\end{proposition}

\begin{proof}
For any $n \in \bN$,
\begin{equation}
 S_{n}(\theta)-S_{n-1}(\theta)= s(\theta;\bv_{n}, X_{n})= \sum_{k=1}^{m} \bigl(a_{nk}- \bar{a}_{n}(\theta; \bv_{n}) \bigr)  \ind{X_{n}=k}.
 \end{equation}
 Therefore,  $| S_{n}(\theta)-S_{n-1}(\theta)| \leq 2K$ for every $n \in \bN$.
Moreover, since  $\bv_{n}$ is an $\cF_{n-1}$-measurable random vector,  it follows directly from  \eqref{score}
\begin{align*}
\Exp_{\theta} [S_{n}(\theta)-S_{n-1}(\theta)|\cF_{n-1}] &= \Exp_{\theta}[ s(\theta;\bv_{n}, X_{n}) | \cF_{n-1}]=0,
\end{align*}
which proves the martingale property of $S_{n}(\theta)$. Next, from   \eqref{fisher} it follows that
\begin{align*}
\Exp_{\theta} [ (S_{n}(\theta)-S_{n-1}(\theta) )^{2}|\cF_{n-1}] &= \Exp_{\theta}[ s^{2}(\theta;\bv_{n}, X_{n}) | \cF_{n-1}]= J(\theta;\bv_{n}),
\end{align*}
which proves that $\langle S(\theta) \rangle_{n}= \sum_{i=1}^{n}  J(\theta;\bv_i).$
Finally, from \eqref{derivative} it follows that for any $\tilde{\theta}$ we have
\begin{align*}
S^{'}_n(\tilde{\theta}) &:= \frac{d}{d\theta} S_n(\theta) \Big|_{\theta=\tilde{\theta}}= \sum_{i=1}^{n} -J(\tilde{\theta};\bv_{i})=  -I_n(\tilde{\theta}),
\end{align*}
which completes the proof.
\end{proof}

With the  next theorem we establish  the strong consistency of $\hat{\theta}_n$ for any item selection strategy.

 \begin{thm} \label{theo1}
 For any  item selection strategy,  as $n \rightarrow \infty$ we have
 \begin{equation} \label{consistent}
\hat{\theta}_n\rightarrow \theta \quad \text{and} \quad
\frac{I_n(\hat{\theta}_{n})}{I_{n}(\theta)}   \rightarrow1  \quad \Pro_{\theta}-\text{a.s.}
\end{equation}
\end{thm}

\begin{proof}
Let    $(\bv_n)_{n \in \bN}$ be an arbitrary   item selection strategy. From Proposition \ref{prop1} it follows that  $S_n(\theta)$ is a $\Pro_\theta$-martingale with mean $0$ and predictable variation  $I_n(\theta)\geq nJ_{*}(\theta)\rightarrow\infty$, since  $J_{*}(\theta)>0$.  Then, from the  Martingale Strong Law of Large Numbers (see, e.g., \cite{willi}, p. 124),  it follows that as $n \rightarrow \infty$
\begin{align}\label{score_CAT}
\frac{S_n(\theta)}{I_n(\theta)}\rightarrow 0 \quad \Pro_\theta-\mbox{a.s.} .
\end{align}
From a Taylor expansion of  $S_{n}(\theta)$ around  $\hat{\theta}_{n}$  it follows that there  exists some $\tilde{\theta}_n$ that  lies between $\hat{\theta}_n$ and $\theta$ so that
\begin{align}\label{Taylor}
\begin{split}
0=S_n(\hat{\theta}_n) &=S_n(\theta)+S^{'}_n(\tilde{\theta}_n)(\hat{\theta}_n-\theta) = S_n(\theta)- I_n(\tilde{\theta}_n)(\hat{\theta}_n-\theta),
\end{split}
\end{align}
where  the second equality follows from \eqref{deriv2}.  From \eqref{score_CAT} and \eqref{Taylor} we then obtain
\begin{align*}
\frac{I_n(\tilde{\theta}_n)}{I_n(\theta)} \, (\hat{\theta}_n-\theta)\rightarrow 0 \quad \Pro_\theta-\mbox{a.s.}
\end{align*}
The strong consistency of $\hat{\theta}_n$ will then follow as long as we can guarantee that the fraction in the last relationship remains bounded away from $0$ as $n \rightarrow \infty$. However, for every $n$ we have
\begin{align*}
\frac{I_n(\tilde{\theta}_n)}{I_n(\theta)}=\frac{\sum_{i=1}^{n}J(\tilde{\theta}_n;\bv_{i})}{\sum_{i=1}^{n}J(\theta;\bv_i)}\geq \frac{n J_{*}(\tilde{\theta}_n)}{n J^{*}(\theta)} = \frac{J_{*}(\tilde{\theta}_n)}{J^{*}(\theta)}.
\end{align*}
Since $J^{*}(\theta)>0$, it suffices to show   that
$\Pro_\theta(\liminf_n J_{*}(\tilde{\theta}_n)>0)=1$. Since $J_{*}(\theta)$ is continuous,  positive and bounded away from 0 when $|\theta|$ is bounded away from infinity (recall \eqref{limiting2}) and $\tilde{\theta}_n$ lies between $\hat{\theta}_n$ and $\theta$, it suffices to show that

\begin{align} \label{prove}
\Pro_\theta(\limsup_n  |\hat{\theta}_n|>0)=1.
\end{align}
In order to prove \eqref{prove}, we observe first of all that since
 $S_n(\hat{\theta}_n)=0$ for large $n$,  \eqref{score_CAT} can be rewritten as follows:
 \begin{align}\label{proof_tool1}
\frac{S_n(\theta)-S_n(\hat{\theta}_n)}{I_n(\theta)}  \rightarrow 0\quad \Pro_\theta-\text{a.s.}
\end{align}
But for every $n$  we have $I_n(\theta) \leq  n J^{*}(\theta)$ and
\begin{align*}
\begin{split}
S_n(\theta)-S_n(\hat{\theta}_n)
&=\sum_{i=1}^{n} \left[s(\theta;\bv_i, X_i)-s(\hat{\theta}_n
;\bv_i, X_i) \right]   \\
&= \sum_{i=1}^{n}  \left[   \bar{a}(\hat{\theta}_n; \bv_{i})- \bar{a}(\theta;\bv_{i}) \right]   \geq n \inf_{\bv \in \bB} \left[  \bar{a}(\hat{\theta}_n; \bv)- \bar{a}(\theta;\bv) \right],
\end{split},
\end{align*}
therefore we obtain
\begin{align}\label{proof_tool3}
\frac{S_n(\theta)-S_n(\hat{\theta}_n)}{I_n(\theta)}
&\geq \frac{\inf_{\bv \in \bB}  \left[  \bar{a}(\hat{\theta}_n; \bv)- \bar{a}(\theta;\bv) \right]}{ J^{*}(\theta)}.
\end{align}


On the event $\{\limsup_n \hat{\theta}_n =\infty\}$ there exists a subsequence $(\hat{\theta}_{n_{j}})$ of $(\hat{\theta}_{n})$ such that $\hat{\theta}_{n_{j}} \rightarrow \infty$. Consequently, for any $\bv \in \bB$  we have
\begin{align}\label{proof_tool4}
\lim_{n_{j} \rightarrow \infty}  \left[  \bar{a}(\hat{\theta}_{n_{j}};\bv)- \bar{a}(\theta;\bv) \right] =   a^{*}(\bv)  - \bar{a}(\theta;\bv) >0.
\end{align}
Since $a^{*}(\bv) - \bar{a}(\theta;\bv)$ is jointly continuous in $\theta$ and $\bv$, from  Lemma \ref{lem1}  we obtain
\begin{align} \label{proof_tool5}
\liminf_{n_{j} \rightarrow \infty}  \; \inf_{\bv \in \bB}  \left[  \bar{a}(\hat{\theta}_{n_{j}};\bv)- \bar{a}(\theta;\bv) \right] &\geq \inf_{\bv \in \bB}  \left[ \bar{a}(\theta;\bv) - \bar{a}(\theta;\bv) \right] >0.
\end{align}
From  \eqref{proof_tool3} and  \eqref{proof_tool5} it follows that
\begin{align*}
\liminf_{n_{j} \rightarrow \infty}   \frac{S_{n_j}(\theta)-S_{n_j}(\hat{\theta}_{n_{j}})}{I_{n_j}(\theta)}  >0
\end{align*}
and  comparing with   \eqref{proof_tool1} we conclude that
$\Pro_{\theta} ( \limsup_n \hat{\theta}_n =  \infty )=0$. In an identical way we can show that $\Pro_{\theta} ( \liminf_{n} \hat{\theta}_n =  -\infty )=0$, which establishes  \eqref{prove} and completes the proof of the strong consistency of $\hat{\theta}_n$. In order to prove the second part of \eqref{consistent}, we observe that
\begin{align*}
\frac{ | I_n(\hat{\theta}_n)-I_n(\theta)|}{I_{n}(\theta)} &\leq  \frac{1}{n J_{*}(\theta)} \sum_{i=1}^{n} |J(\hat{\theta}_n; \bv_i)- J(\theta; \bv_i)| \\
&\leq \frac{1}{J_{*}(\theta)}   \sup_{\bv\in \bB}|J(\hat{\theta}_n;\bv)-J(\theta;\bv)| .
\end{align*}
But since  $J(\theta;\bv)$ is jointly continuous and $\hat{\theta}_{n}$ strongly consistent, from Lemma \ref{lem1} it follows that
the upper bound goes to 0 almost surely, which completes the proof.
\end{proof}

While the strong consistency of  $\hat{\theta}_n$ could be established  for any item selection strategy, its asymptotic normality and efficiency requires the  information-maximizing item selection strategy \eqref{item_select}.


\begin{thm} \label{theo3}
If the information-maximizing item selection strategy \eqref{item_select} is used,  then $\hat{\theta}_n$ is asymptotically normal   as $n \rightarrow \infty$, since
 \begin{align} \label{clt1}
& \sqrt{I_n( \hat{\theta}_{n}) } \, (\hat{\theta}_n-\theta )  \rightarrow \cN(0,1)
\end{align}
and  asymptotically  efficient, in the sense that
 \begin{align} \label{clt2}
 & \sqrt{n}(\hat{\theta}_n-\theta)  \rightarrow
\cN\left( 0, [ J^{*}(\theta)]^{-1} \right).
\end{align}

\end{thm}

\begin {proof}
We will denote $\{\hat{\bv}_{i}\}_{1 \leq i \leq n}$ as the information-maximizing item selection strategy \eqref{item_select}.  We will start by showing that   as $n \rightarrow \infty$
\begin{align}\label{result0}
 \frac{1}{n} I_n(\theta)=\sum_{i=1}^{n} J(\theta;\hat{\bv}_i)
 & \rightarrow J^{*}(\theta) \quad \Pro_\theta-\text{a.s}.
\end{align}
In  order to do so, it  suffices to show   that $
J(\theta;\hat{\bv}_n) \rightarrow J^{*}(\theta)$  $\Pro_\theta$-a.s.
Since $J(\theta; \bv)$ is jointly continuous and $\hat{\theta}_n$ a strongly consistent estimator of $\theta$, from Lemma \ref{lem1} we have
\begin{equation} \label{ddd}
 |J(\hat{\theta}_n;\hat{\bv}_n) -  J(\theta;\hat{\bv}_n)| \leq  \sup_{\bv \in \bB}   |J(\hat{\theta}_n; \bv) -  J(\theta;\bv)| \rightarrow 0 \quad \Pro_\theta-\text{a.s.}
\end{equation}
Therefore, we only need to show that
$J(\hat{\theta}_n;\hat{\bv}_n) \rightarrow J^{*}(\theta)$ $\Pro_\theta$-a.s.
But from the definition of  $(\hat{\bv}_n)$  in \eqref{item_select} we have that   $J(\hat{\theta}_{n-1}; \hat{\bv}_n)=J^{*}(\hat{\theta}_{n-1})$, therefore  from the triangle inequality we obtain:
\begin{align*}
|J(\hat{\theta}_n;\hat{\bv}_n)-J^{*}(\theta)|&\leq |J(\hat{\theta}_n;\hat{\bv}_n)-J(\hat{\theta}_{n-1};\hat{\bv}_n)|+|J^{*}(\hat{\theta}_{n-1})-J^{*}(\theta)|\\
&\leq\sup_{\bv\in \bB}|J(\hat{\theta}_n;\bv)-J(\hat{\theta}_{n-1};\bv)|+
|J^{*}(\hat{\theta}_{n-1})-J^{*}(\theta)|.
\end{align*}
Since $\hat{\theta}_{n}$ is a strongly  consistent estimator of $\theta$,
from  Lemma \ref{lem1} it follows that
$$
\sup_{\bv\in \bB} |J(\hat{\theta}_n;\bv)-J(\hat{\theta}_{n-1};\bv) | \rightarrow 0 \quad \Pro_\theta-\text{a.s.} ,
$$
whereas from  the continuity of $J^{*}$  we obtain
$J^{*}(\hat{\theta}_{n-1}) \rightarrow J^{*}(\theta)$ $\Pro_\theta-\text{a.s.}$,  which completes the proof of \eqref{result0}.

Now, from Proposition \ref{prop1} we know that  $\{S_n(\theta)\}_{n \in \bN}$ is a  martingale with bounded increments,  mean $0$ and predictable variation $I_n(\theta)$. Then, due to \eqref{result0}, we can apply the  Martingale Central Limit Theorem (see, e.g., \cite{billi}, Ex. 35.19, p. 481) and obtain
$$
\frac{S_n(\theta)}{\sqrt{I_n(\theta)}} \longrightarrow \cN(0,1).
$$
Using the  Taylor expansion \eqref{Taylor},  we  have
\begin{align*}
\frac{I_n( \tilde{\theta}_n)}{I_n(\theta)}\sqrt{I_n(\theta)} \, (\hat{\theta}_n-\theta)\rightarrow \cN(0,1),
\end{align*}
where $\tilde{\theta}_n$ lies between $\hat{\theta}_n$ and $\theta$.
But, from  \eqref{consistent} it follows that
$$
\frac{I_n( \tilde{\theta}_n)}{I_n(\theta)}  \rightarrow 1 \quad \Pro_\theta-\text{a.s.},
$$
thus,  from an application of Slutsky's theorem we obtain
 \begin{align} \label{clt3}
\sqrt{I_n(\theta)} \, (\hat{\theta}_n-\theta)  &\longrightarrow \cN(0,1).
\end{align}
Finally, from \eqref{clt3} and \eqref{result0} we obtain \eqref{clt2}, whereas from
from \eqref{clt3} and \eqref{consistent} we obtain  \eqref{clt1}, which completes the proof.
\end{proof}


\section{CAT with response revision}

In this section we  consider the design of  CAT when response revision is allowed. As before, we consider multiple-choice items that have $m$ categories and we assume that the total number of items that will be administered is fixed and equal to $n$. However, at any time during the test the examinee can go back and revise (i.e., change) the answer to a previous item.  The only  restriction that we impose is that each item can be revised at most $m-2$ times during the test. As a result, we now focus on items  with $m \geq 3$ categories, unlike the previous section where the case of binary items ($m=2$) was also included.  Moreover, due to the possibility of revisions,  the total number of responses (first answers and revisions)  that are observed, $\tau_n$,  is random,  even though the total number of administered items, $n$,  is fixed.  In any case,  $n \leq \tau_n \leq (m-1) n$, with the lower bound corresponding to the case of no revisions and the upper bound to the case that all items are revised as many times as possible.

\subsection{Setup}

In order to formulate the problem in more detail, suppose that at some point during the test we have collected  $t$ responses, let $f_{t}$ be the number of  distinct items that have been administered and $r_t:=t-f_{t}$  the number of revisions.  For each  item $i \in \{1, \ldots,f_{t}\}$, we denote $g_{t}^{i}$ as the  number of responses that correspond to this particular item. Since each item  can be revised up to $m-2$ times, we have
 $1\leq g_{t}^i \leq m-1$.

 After completing the $t^{th}$ response,  the examinee decides whether to  revise one of the previous items or to proceed to a new item. Specifically,  let $ C_{t}:=\{i\in \{1,\ldots, f_{t} \}: g_{t}^{i}<m-1\}$ be the set of items that can still be revised. The decision of the examinee  is then captured by the following random variable:
$$
d_{t} := \begin{cases}
 0, \quad \text{the $t+1^{th}$ response will correspond to a new item}  \\
  i, \quad \text{the $t+1^{th}$ response is a revision of item} \; i  \in C_{t}
 \end{cases} ,
 $$
 with the understanding that $d_{t}=0$ when $C_t=\emptyset$.  Then,
$\cG_{t}:= \sigma \left( f_{1:t} , d_{1:t} \right)$  is the $\sigma$-algebra that  contains all  information regarding the history of revisions, where for
compactness we write $f_{1:t}:= (f_{1}, \ldots, f_{t})$ and
$d_{1:t}:= (d_{1}, \ldots, d_{t})$.

Of course, we also observe the responses of the examinee during the test.  For each  item $i \in \{1, \ldots,f_{t}\}$,  let $A^{i}_{j-1}$  be the set of remaining  categories  just before the  $j^{th}$ attempt on this particular item, where $1 \leq j \leq  g_t^i$.  Thus,  $A^{i}_{0}=\{1, \ldots, m\}$ is the set of all categories and $A^{i}_{j-1}$ is a random set for $j >1$. Let $X_{j}^{i}$ be the response that corresponds to the   $j^{th}$ attempt, so that  $X_{j}^{i} = k$ if category $k$ is chosen  on the $j^{th}$ attempt on item $i$, where $k \in A^{i}_{j-1}$.  Then,
$$
\cF^X_{t}:= \sigma  \left( X_{1:g_t^i}^i, \;   \;  1 \leq  i \leq f_{t} \right), \quad \text{where} \quad X_{1:g_t^i}^i:=(X^i_1, \ldots, X^i_{g_t^i} ),
$$
is the $\sigma$-algebra that captures the information  from the observed responses and $ \cF_{t}:= \cG_{t} \vee \cF^X_{t}$ the $\sigma$-algebra that contains all the  available information up to this time.

\subsection{Modeling assumptions}
As in the case of the regular CAT that we considered in the previous section,
we assume that the first  response to each item is governed by the  nominal response model, so that for every $1 \leq i \leq f_t$ we have
 \begin{align} \label{nominal3}
 \Pro_{\theta} (   X^{i}_{1} =k  \, | \, \bv_{i}  ) &:= p_{k}(\theta; \bv_i), \quad
1 \leq k \leq m,
\end{align}
where   $\bv_i:=(a_{i2},\ldots, a_{im}, c_{i2},\ldots, c_{im})$  is a $\bB$-valued  vector that characterizes item $i$ and satisfies \eqref{trivial}-\eqref{identify}  and  $p_{k}(\theta; \bv_i)$ is the pmf of the nominal response model defined in \eqref{nominal2}. But we now further assume that revisions are  also governed  by the nominal response model, so that for every $2 \leq j \leq g_t^i$  we have
 \begin{align} \label{nominal4}
 \Pro_{\theta} \left( X_{j}^{i} =k \, | \,  X_{1:j-1}^{i} , \bv_{i}  \right)
&:= \frac{p_{k}(\theta; \bv_i)}{\sum_{h \in A^i_{j-1}} p_{h}(\theta; \bv_i)}, \quad k \in A^i_{j-1},
\end{align}
where  $X^i_{1:j} := (X_{1}^i, \ldots, X^i_{j})$. Moreover, we assume, as in the previous section, that responses coming from different items are conditionally independent, so that
 \begin{align} \label{independence2}
\Pro_{\theta} \left(   X^{i}_{1:g_{t}^{i}}  \; , \; 1 \leq  i \leq   f_{t} \,  | \,  d_{1:t}, \bv_{1: f_{t}}  \right) = \prod_{i=1}^{f_{t}}  \Pro_{\theta} \left(   X^{i}_{1:g_{t}^{i}} \, | \,  d_{1:t} , \bv_{i}  \right) ,
\end{align}
where for compactness we write $\bv_{1:f_t}:= (\bv_{1}, \ldots, \bv_{f_{t}})$. Finally, we additionally assume that the observed responses on any given item are  conditionally independent of the time during the test at which they were given. In other words,  for every $1 \leq i \leq f_t$ we have
\begin{align} \label{stationarity}
 \Pro_{\theta} \left(   X^{i}_{1:g_{t}^{i}} \, | \,  d_{1:t}, \bv_{i}  \right)
&= \Pro_{\theta} (   X^{i}_{1} \, | \, \bv_{i}  )     \cdot  \prod_{j=2}^{g_{t}^{i}}  \Pro_{\theta} \left( X_{j}^{i} \, | \,   X_{1:j-1}^{i} , \bv_{i}  \right) .
\end{align}

The above assumptions specify completely the probability in the left-hand side of \eqref{independence2}. Specifically, \eqref{independence2} and  \eqref{stationarity}  imply that
 \begin{align} \label{partial}
 \begin{split}
& \Pro_{\theta} \left(   X^{i}_{1:g_{t}^{i}}  \; , \; 1 \leq  i \leq   f_{t} \,  | \,  d_{1:t}, \bv_{1: f_{t}}  \right)   \\
&= \prod_{i=1}^{f_{t}}  \Pro_{\theta} \left(   X^{i}_{1} \, |  \bv_{i}  \right)    \prod_{j=2}^{g_{t}^{i}}  \Pro_{\theta} \left( X_{j}^{i} \, | \,   X_{1:j-1}^{i} , \bv_{i}  \right)
\end{split}
\end{align}
and the probabilities in the right-hand side are determined by the nominal response model according to \eqref{nominal3}-\eqref{nominal4}.  On the other hand, we do \textit{not} model  the decision of the examinee whether to revise or not at each step, i.e.,  we do not specify  $\Pro_{\theta} \left(   d_{1:t} \, | \,  \bv_{1: f_{t}} \right)$. While this probability may depend on $\theta$ and provide useful information for the ability of the examinee, its specification is a rather difficult task.  Nevertheless, the  above assumptions will be sufficient for the design and analysis of CAT that allows for response revision.

\subsection{Problem formulation}

As we mentioned in the beginning of the section, the  total number of administered items  is fixed and will be denoted  by $n$, as in the case of the regular CAT. However, due to the possibility of revision, the total number of responses will now be random and  denoted by $\tau_n$. Indeed, the test will stop when $n$ items have been distributed \textit{and} the examinee does not want to (or cannot) revise any more items. More formally,
$$
\tau_n:=\min\{t \geq 1: f_{t}=n \;  \text{and} \;  d_{t}=0\},
$$
which reveals that  $\tau_n$ is a  stopping time with respect to filtration $\{ \cG_{t}\}$, and of course $\{ \cF_{t}\}$.  Note that, for every $1 \leq i \leq n-1$,
$$
\tau_i:=\min\{t \geq 1: f_{t}=i \;  \text{and} \;  d_{t}=0\},
$$
is the time at which the $(i+1)^{th}$ item needs to be selected and its selection will depend on the  available information up to this time. That is, we will now say that  $(\bv_{i+1})_{1 \leq i \leq n-1}$ is an  \textit{item selection strategy} if the
parameter vector that characterizes the $(i+1)^{th}$ item, $\bv_{i+1}$,  is a $\bB$-valued, $\cF_{\tau_{i}}$-measurable  random vector. As in the case of the standard CAT, items need to be selected so that the accuracy of the
final  estimator of $\theta$, $\hat{\theta}_{\tau_n}$, be maximized. As in the previous section, a  reasonable approach is to select the items in order to maximize the  Fisher information of the nominal response model at the current ability estimate. Thus,  after each observation $t$ until the end of the test, we need   an  $\cF_{t}$-measurable  random variable,  $\hat{\theta}_{t}$, that will provide the current  estimate for the ability parameter, $\theta$.

\subsection{Adaptive ability estimation  based on a partial likelihood}

Our estimate for  $\theta$ after the first $t$ responses  will be the maximizer of the conditional log-likelihood of the acquired observations given the selected items and the revision strategy of the examinee:
\begin{align}
L_t (\theta)&:= \log \Pro_{\theta} \left(   X^{i}_{1:g_{t}^{i}} \, , \;  i=1, \ldots, f_{t} \,  \Big| \,  d_{1:t}, \; \bv_{1: f_{t}}  \right).
\end{align}
In order to lighten the notation,  for every $2 \leq j \leq g_t^i$
we will use the following notation
\begin{equation} \label{condi_nominal}
p_{k}(\theta; \bv_i \,| \, X^i_{1:j-1}) := \Pro_{\theta} \left( X_{j}^{i} =k \, | \,  X_{1:j-1}^{i} , \bv_{i}  \right), \quad k \in A^i_{j-1}
\end{equation}
 for the conditional probability  that is determined in \eqref{nominal3} and we will further use the following notation for the corresponding log-likelihood
 $$
\ell \left(\theta;\bv_i, X_{j}^{i}=k \,| \, X^{i}_{1:j-1} \right) :=   \log  p_{k}\left(\theta; \bv_i \, | \, X_{1:j-1}^i \right) ,  \quad k \in A^i_{j-1}.
$$
 Then  from \eqref{partial} we have
\begin{align}
\begin{split}
L_t (\theta)&= \sum_{i=1}^{f_t}  \Bigl[ \ell \left(\theta;\bv_i, X^i_{1} \right) +   \sum_{j=2}^{g_{t}^i} \ell(\theta;\bv_i, X_{j}^{i}|X^{i}_{1:j-1})  \Bigr],
\end{split}
\end{align}
where $\ell (\theta;\bv_i, X^i_{1})$ is defined according to \eqref{loglikeli} and the    corresponding score function takes the form
\begin{equation} \label{rev_score}
S_t(\theta) :=\frac{d}{d\theta} L_t(\theta)=
\sum_{i=1}^{f_{t}}  \Bigl[ s\left(\theta;\bv_i, X^i_{1} \right)+
 \sum_{j=2}^{g_{t}^i} s\left(\theta; \bv_i,X^i_j|X^{i}_{1:j-1} \right) \Bigr],
\end{equation}
where  $s (\theta;\bv_i, X^i_{1})$ is defined according to \eqref{score}
and   for every $2 \leq j \leq g_t^i$ we have
\begin{align*}
s(\theta;\mathbf{b}_i,X^i_j=k \, | \, X^{i}_{1:j-1}) &:= \frac{d}{d\theta}
\ell \left(\theta;\bv_i, X_{j}^{i}=k | \, X^{i}_{1:j-1} \right)  \\
& = \sum_{k \in A^i_{j-1}}  \Bigl( a_{ki} - \bar{a}(\theta;\bv_i|X^i_{1:j-1}) \Bigr) \, \ind{X^i_{j}=k} , \quad k \in A^i_{j-1} \\
\quad \text{and} \quad \bar{a}(\theta;\bv_i|X^i_{1:j-1}) &:=\sum_{k \in A^i_{j-1}} a_{ki} \;\ p_ {k}\left(\theta; \bv_i \, | \, X_{1:j-1}^i \right).
\end{align*}
Our estimate for $\theta$ after the first $t$ responses will be the root of $S_{t}(\theta)$. As in the case of the regular CAT, this root will exist  for every $t> t_{0}$, where $t_0$ is some random time. Thus, for $t \leq t_0$ we need an alternative estimating scheme. This,  however,  will not affect the asymptotic properties of our estimator as the  number of administered items, $n$, goes to infinity, which will be the focus on the remaining of this section.

\subsection{Asymptotic analysis}
Our asymptotic analysis will be based on the  martingale property of the score function, $S_{t}(\theta)$, which is established in the following proposition.

\begin{proposition}\label{prop2}
For any item selection strategy and any revision strategy, $\{S_{t}(\theta)\}_{t \in \bN}$ is a $(\Pro_{\theta},\{\cF_{t} \}_{t \in \bN})$-martingale with bounded increments, mean zero and  predictable variation
$\langle S(\theta)\rangle_t = I_t(\theta)$, where
\begin{equation} \label{new_fishern}
I_t(\theta) : =\sum_{i=1}^{f_{t}}
J(\theta;\bv_{i})+  I^{R}_t(\theta), \quad
 I^{R}_t(\theta)  := \sum_{i=1}^{f_{t}}   \sum_{j=2}^{g_t^i} J \left(\theta;\bv_{i} \, | \, X_{1:j-1}^i \right)  ,
\end{equation}
where  $ J(\theta;\bv_{i} )$ is defined in \eqref{fisher} and
\begin{align*}
J(\theta;\bv_i |X^i_{1:j-1} ) &:= \Exp_{\theta}[ s^{2}(\theta; \bv_i, X^i_{j} |X^i_{1:j-1})]\\
&= \sum_{k\in A^i_{j-1}} \Bigl(a_k-\bar{a}(\theta;\bv_i \, |X^i_{1:j-1}) \Bigr)^2 \; p_k(\theta;\bv_i \,  |X^i_{1:j-1}).
\end{align*}
Finally, for any $\tilde{\theta}$ we have
\begin{equation} \label{deriv}
S'(\tilde{\theta})=\frac{d}{d\theta} S_{t}(\theta)\Big|_{\theta=\tilde{\theta}}=-I_t(\tilde{\theta}).
\end{equation}

\end{proposition}

\begin{proof}
After having completed the $t-1^{th}$ response,  the examinee  either  proceeds with a new item or chooses to revise a previous item. Therefore,
the difference $S_{t}(\theta)- S_{t-1}(\theta)$ admits the following decomposition:
\begin{align} \label{decompose}
&s \left(\theta; \bv_{f_t}, X_{1}^{f_{t}} \right)  \, \ind{d_{t-1}=0} +     \sum_{i \in C_{t-1}} s\left(\theta; \bv_{i}, X^{i}_{g_{t}^{i}} | X^i_{1:g_{t}^{i}-1} \right) \, \ind{d_{t-1}=i} ,
\end{align}
where the sum is null when $C_{t-1} =\emptyset$.  Since $d_{t-1}, C_{t-1}$ are $\cF_{t-1}$-measurable, taking conditional expectations  with respect to $\cF_{t-1}$ we obtain
\begin{align*}
\Exp_{\theta}[S_{t}(\theta)- S_{t-1}(\theta) |\cF_{t-1}] &=
\Exp_{\theta} \left[ s \left(\theta; \bv_{f_t}, X_{1}^{f_{t}}  \right)  \, \Big|\, \cF_{t-1} \right]  \, \ind{d_{t-1}=0} \\
& +  \sum_{i \in C_{t}}  \Exp_{\theta} \left[   s\left(\theta; \bv_{i}, X^{i}_{g_{t}^{i}} | X^i_{1:g_{t}^{i}-1} \right)  \,   \Big| \, \cF_{t-1} \right] \, \ind{d_{t-1}=i}.
\end{align*}
Since  $f_{t}$ is $\cF_{t-1}$-measurable, it follows that
$$
\Exp_{\theta} \left[ s \left(\theta; \bv_{f_t}, X_{1}^{f_{t}}  \right)  \, \Big|\, \cF_{t-1} \right]=0
$$
and since $g_{t}^{i}$ is also  $\cF_{t-1}$-measurable , it follows that
$$
\Exp_{\theta} \left[   s\left(\theta; \bv_{i}, X^{i}_{g_{t}^{i}} | X^i_{1:g_{t}^{i}-1} \right)  \,   \Big| \, \cF_{t-1} \right]=0,
$$
which proves that $S_{t}(\theta)$ is a zero-mean martingale with respect to $(\Pro_{\theta},\{\cF_{t} \}_{t \in \bN})$.  Now, taking squares in \eqref{decompose} we obtain
\begin{align*}
& \Exp_{\theta}[ \left(S_{t}(\theta)- S_{t-1}(\theta) \right)^{2}|\cF_{t-1}] \\
&= J( \theta; \bv_{f_{t}})  \, \ind{d_{t-1}=0} +  \sum_{i \in C_{t-1}}   J \left(\theta; \bv_{i} | X^i_{1:g_{t}^{i}-1} \right)  \, \ind{d_{t-1}=i}
\end{align*}
and consequently the predictable variation of $S_t(\theta)$ will be
\begin{align*}
\langle S(\theta)\rangle_t &= \sum_{v=1}^{t} \Exp_{\theta} \left[ \left(S_{v}(\theta)- S_{v-1}(\theta) \right)^{2}|\cF_{v-1} \right]  \\
&=  \sum_{v=1}^{t} \left[  J( \theta; \bv_{f_{v}})  \, \ind{d_{v-1}=0} +  \sum_{j \in C_{v-1}}   J \left(\theta; \bv_{j} | X^{j}_{1:g_{v-1}^{j}} \right)  \, \ind{d_{v-1}=j} \right] \\
&=\sum_{i=1}^{f_{t}}  \left[  J(\theta;\bv_{i})
+  \sum_{h=2}^{g_{t}^i}  J(\theta;\bv_{i},h) \right]=: I_{t}^{R}.
\end{align*}
\end{proof}

We can now establish the strong consistency of $\hat{\theta}_{\tau_n}$ as $n \rightarrow \infty$ without any conditions on the item selection or the revision strategy.

\begin{thm} \label{LLNCAT}
For any item selection method and any revision strategy,  as $n \rightarrow \infty$ we have
\begin{equation} \label{revi_consistent}
  \hat{\theta}_{\tau_n} \rightarrow \theta \quad \text{and} \quad
 \frac{I_{\tau_n}(\hat{\theta}_{\tau_n})}{I_{\tau_{n}}(\theta)}   \rightarrow 1
\quad \Pro_{\theta}\text{-a.s.}
\end{equation}
\end{thm}

\begin{proof}
From Proposition \ref{prop2} we have that  $S_{t}(\theta)$ is a $(\Pro_\theta, \{\cF_{t} \} )$-martingale.   Moreover,  $(\tau_n)_{n \in \bN}$ is a strictly  increasing  sequence of  (bounded) $\{\cF_{t} \} )$-stopping times.   Then, from an application of the Optional Sampling Theorem it follows that
$S_{\tau_n}(\theta)$ is a $(\Pro_\theta, \{\cF_{\tau_n} \} )$-martingale with predictable variation $I_{\tau_n}(\theta)$. Moreover,
from \eqref{new_fishern} we have  $I_{\tau_n}(\theta) \geq n J_{*}(\theta)\rightarrow\infty$, therefore from the Martingale Strong Law of Large Number (\cite{willi}, p. 124 ) it follows that
\begin{align}\label{scoreconverge}
\frac{S_{\tau_n}(\theta)}{I_{\tau_n}(\theta)}\rightarrow 0 \quad \Pro_\theta-\text{a.s.}
\end{align}
Then, by a Taylor expansion around $\theta$ and \eqref{deriv} we have
\begin{align} \label{revi_taylor}
\begin{split}
 0= S_{\tau_n}(\hat{\theta}_{\tau_n}) &=   S_{\tau_n}(\theta)+
 S^{'}_{\tau_n}(\tilde{\theta}_{\tau_n})(\hat{\theta}_{\tau_n}-\theta)  \\
&=S_{\tau_n}(\theta) -I_{\tau_n}(\tilde{\theta}_{\tau_n})(\hat{\theta}_{\tau_n}-\theta),
\end{split}
\end{align}
where $\tilde{\theta}_{\tau_n}$ lies between $\hat{\theta}_{\tau_n}$ and $\theta$, and   \eqref{scoreconverge} takes the form
\begin{align*}
\frac{I_{\tau_n}(\tilde{\theta}_{\tau_n})}{I_{\tau_{n}}(\theta)} \, (\hat{\theta}_{\tau_n}-\theta)\rightarrow 0 \quad \Pro_\theta-\text{a.s.}
\end{align*}
However, since $\tau_n \leq (m-1)n$ and $\, J_{*}(\theta)  f_t \leq I_{t}(\theta) \leq K t$ for every $t$, where $K$ is defined in \eqref{K}, we have
$$
\frac{I_{\tau_n}(\tilde{\theta}_{\tau_n})}{I_{\tau_{n}}(\theta)} \geq
\frac{n J_{*}(\tilde{\theta}_{\tau_n})}{\tau_{n} K } \geq \frac{1}{m-1} J_{*}(\tilde{\theta}_{\tau_n})
$$
and it suffices to show that
\begin{align}\label{bound_RCAT}
\limsup_n |\hat{\theta}_{\tau_n}|<\infty \quad \Pro_\theta-\text{a.s.}
\end{align}
Now, for large $n$ we have $S_{\tau_n}(\hat{\theta}_{\tau_n})=0$  and  \eqref{scoreconverge} can be rewritten as follows
\begin{align} \label{look}
\frac{S_{\tau_n}(\theta)-S_{\tau_n}(\hat{\theta}_{\tau_n})} {I_{\tau_n}(\theta)}
&\rightarrow 0\quad \Pro_\theta-\text{a.s.}
\end{align}
But from the definition of the score function in \eqref{rev_score}
it follows that
\begin{align*}
& S_{\tau_n}(\theta)-S_{\tau_n}(\hat{\theta}_{\tau_n})  \\
&=  \sum_{i=1}^{n}   \left[ \left( s(\theta;\bv_i)- s(\hat{\theta}_{\tau_n};\bv_i) \right)+  \sum_{j=2}^{g_{\tau_n}^{i}}
\left( s(\theta;\bv_i, X^i_j|X^{i}_{1:j-1})-s(\hat{\theta}_{\tau_n};\bv_i, X^{i}_{j}|X^{i}_{1:j-1}) \right)  \right]\\
&= \sum_{i=1}^{n}   \left[ \left( \bar{\alpha}(\hat{\theta}_{\tau_n};\bv_i)-\bar{\alpha}(\theta;\bv_i) \right)
+   \sum_{j=2}^{g_{\tau_n}^{i}}
\left( \bar{\alpha}(\hat{\theta}_{\tau_n};\bv_i|X^{i}_{1:j-1})-\bar{\alpha}(\theta;\bv_i|X^{i}_{1:j-1}) \right)  \right]\\
&\geq n \, \inf_{\bv\in \bB} [\bar{\alpha}(\hat{\theta}_{\tau_n};\bv)-\bar{\alpha}(\theta;\bv)] \\
&+ (\tau_n-n) \min_{2 \leq j \leq m-1} \min_{X^i_{1:j-1}}  \inf_{\bv\in \bB }[ \bar{\alpha}(\hat{\theta}_{\tau_n};\bv\, |\, X^i_{1:j-1})-\bar{\alpha}(\theta;\bv|X^i_{1:j-1})] .
\end{align*}
On the other hand,  $I_{\tau_n}(\theta) \leq  \tau_n  K$, which implies that
\begin{align*}
\frac{S_{\tau_n}(\theta)-S_{\tau_n}(\hat{\theta}_{\tau_n})} {I_{\tau_n}(\theta)} &\geq  \frac{1}{K}\inf_{\bv\in \bB} [\bar{\alpha}(\hat{\theta}_{\tau_n};\bv)-\bar{\alpha}(\theta;\bv)] \\
&+ \frac{1}{K}  \min_{2 \leq j \leq m-1} \, \min_{X_{1:j-1}}  \, \inf_{\bv\in \bB } \left[ \bar{\alpha}(\hat{\theta}_{\tau_n};\bv\, |\, X_{1:j-1})-\bar{\alpha}(\theta;\bv|X_{1:j-1}) \right] ,
\end{align*}
where $X_{1:j-1}:= (X_{1}, \ldots, X_{j-1})$ is a vector of $j-1$ responses on an item with parameter $\bv$.  Then, on the event $\{\limsup_{n} \hat{\theta}_{\tau_n} \rightarrow \infty \}$ there exists a subsequence $(\hat{\theta}_{\tau_{n_j}})$ of $(\hat{\theta}_{\tau_n})$ so that  $\hat{\theta}_{\tau_{n_j}} \rightarrow \infty$ and, consequently,
\begin{align*}
\liminf_{n_j \rightarrow \infty}  \, \inf_{\bv\in \bB} \left[\bar{\alpha}(\hat{\theta}_{\tau_{n_j}};\bv)-\bar{\alpha}(\theta;\bv) \right] &>0
\end{align*}
whereas  for any  $2 \leq j \leq m-1$ and $X_{1:j-1}$  we have
\begin{align*}
\liminf_{n_j \rightarrow \infty}  \, \inf_{\bv\in \bB} \left[\bar{\alpha}(\hat{\theta}_{\tau_{n_j}};\bv \, |\, X_{1:j-1})-\bar{\alpha}(\theta;\bv \, |\, X_{1:j-1}) \right] &\geq 0.
\end{align*}
Therefore, we  conclude that
\begin{align*}
\liminf_{n_j} \frac{S_{\tau_{n_j}}(\theta)-S_{\tau_{n_j}}(\hat{\theta}_{\tau_{n_j}})} {I_{\tau_{n_j}}(\theta)} &>0
\end{align*}
and  comparing with  \eqref{look} we have that $\Pro(\limsup_{n} \hat{\theta}_{\tau_n} = \infty)=0$. Similarly we can show that $\Pro(\limsup_{n} \hat{\theta}_{\tau_n} = -\infty)=0$, which proves
\eqref{bound_RCAT} and, consequently,   the strong consistency of  $\hat{\theta}_{\tau_{n}}$ as $n \rightarrow \infty$.   In order to prove the second claim of the theorem,  we need to show that
\begin{equation} \label{second_part}
\frac{| I_{\tau_{n}}(\hat{\theta}_{\tau_{n}})-I_{\tau_{n}}(\theta)|}{I_{\tau_{n}}(\theta)}
\end{equation}
goes to 0 $\Pro_{\theta}$-\text{a.s.}  as $n \rightarrow \infty$. But $I_{\tau_{n}}(\theta) \geq n \, J_{*}(\theta)$, whereas $| I_{\tau_{n}}(\hat{\theta}_{\tau_{n}})-I_{\tau_{n}}(\theta)|$ is bounded above by
\begin{align*}
&  \sum_{i=1}^{n}   |J(\hat{\theta}_{\tau_{n}}; \bv_i)- J(\theta; \bv_i)|
+ \sum_{i=1}^{n} \sum_{j=2}^{g^i_{\tau_n}}
\Big| J(\hat{\theta}_{\tau_{n}}; \bv_i | X^i_{1:j-1})  - J(\theta; \bv_i |  X^i_{1:j-1} )  \Big| \\
&\leq n  \, \sup_{\bv\in \bB} |J(\hat{\theta}_{\tau_{n}}; \bv)-J(\theta;\bv)| \\
&+ (\tau_n-n) \, \max_{2 \leq j \leq m-1} \; \max_{X_{1:j-1} } \; \sup_{\bv \in \bB}    \Big| J(\hat{\theta}_{\tau_{n}}; \bv|X_{1:j-1}) - J(\theta; \bv |X_{1:j-1})\Big| ,
\end{align*}
where again $X_{1:j-1}:= (X_{1}, \ldots, X_{j-1})$ is a vector of $j-1$ responses on an item with parameter $\bv$.
Therefore, the ratio in \eqref{second_part} is bounded above by
\begin{align*}
&\frac{1}{J_{*}(\theta)}  \sup_{\bv\in \bB}|J(\hat{\theta}_t;\bv)-J(\theta;\bv)|\\
& +   \frac{m-2}{J_{*}(\theta)}   \max_{2 \leq j \leq m-1} \; \max_{X_{1:j-1} } \, \sup_{\bv \in \bB}    \Big| J(\hat{\theta}_{\tau_{n}}; \bv|X_{1:j-1}) - J(\theta; \bv |X_{1:j-1}) \Big| .
\end{align*}
But  we can show as in Theorem \ref{theo1} that
$$ \sup_{\bv\in \bB} \,  | J(\hat{\theta}_{\tau_{n}};\bv)-J(\theta;\bv)|  \rightarrow 0
 \quad \Pro_{\theta}-\text{a.s.}
$$
and, similarly,  due to the strong consistency of $\hat{\theta}_{\tau_{n}}$ and the continuity of $\theta \rightarrow J(\theta, \bv\, | \, X_{1:j-1})$, we can apply  Lemma \ref{lem1} and show that  for every $2 \leq j \leq m-1$ we have
$$
\sup_{\bv\in \bB}  \Big| J(\hat{\theta}_{\tau_{n}} ; \bv \, | \, X_{1:j-1}) - J(\theta; \bv \, |\, X_{1:j-1}) \Big| \rightarrow 0
 \quad \Pro_{\theta}-\text{a.s.}
$$
This implies that \eqref{second_part}  goes to 0 a.s. and completes the proof.
\end{proof}

While we established  the  strong consistency of $\hat{\theta}_{\tau_n}$ without any conditions, its  asymptotic normality requires certain conditions on the item selection strategy  and the number of revisions.  Indeed, in order to  apply the Martingale Central Limit theorem, as we did in the case of the regular CAT, we need to make sure that
\begin{equation} \label{decompose3}
\frac{1}{n} I_{\tau_n}(\theta)=  \frac{1}{n} \sum_{i=1}^{n}J(\theta,\bv_i)+
\frac{1}{n} I^{R}_{\tau_n}(\theta)
\end{equation}
converges in probability, where $I^{R}$ is the part of the Fisher information due to revisions (recall \eqref{new_fishern}). If we select each item in order to maximize the  Fisher information at the current estimate of the ability level, i.e.,
\begin{align} \label{item_select2}
\hat{\bv}_{i+1} \in \argmax_{\bv\in \bB} J(\hat{\theta}_{\tau_{i}};\bv),\quad i=1, \ldots, n-1,
\end{align}
where $J(\theta;b)$ is the Fisher information of the nominal response model, defined in \eqref{fisher}, then  we can show as in the case of the regular CAT that
\begin{align*}
\frac{1}{n} \sum_{i=1}^{n}J(\theta, \hat{\bv}_i)\rightarrow J^{*}(\theta) \quad \Pro_\theta-a.s.
\end{align*}
However,  the item selection strategy does not control the second term in
\eqref{decompose3}. Nevertheless,  we can see that
$$
\frac{1}{n} I^{R}_{\tau_n}(\theta) \leq  \frac{K \, (\tau_n-n)}{n},
$$
which implies that  $I^{R}_{\tau_n}(\theta)/n$ will converge to 0 in probability as long as  the number of revisions is small relative to the total number of items, in the sense that  $(\tau_n-n)/n$ goes to 0 in probability, i.e.,  $\tau_n-n=o_p(n)$. This is the content of the following theorem.

\begin{thm}\label{CLTRCAT}
If  $I_{\tau_n}(\theta)/n$ converges in probability, then
\begin{align} \label{rev_asy_nor}
\sqrt{I_{\tau_n}(\hat{\theta}_{\tau_n})} \, (\hat{\theta}_{\tau_n}-\theta) & \rightarrow \cN(0,1).
\end{align}
This is true in particular when the information-maximizing  item selection strategy \eqref{item_select2} is used and the number of revisions is much smaller than the number of items, in the sense that $\tau_n-n=o_p(n)$, in which case we have
\begin{align} \label{rev_asy_nor2}
\sqrt{n}(\hat{\theta}_{\tau_n}-\theta) & \rightarrow \cN\left( 0, [ J^{*}(\theta)]^{-1} \right).
\end{align}
\end{thm}

\begin{proof}
We will first show that  if $I_{\tau_n}(\theta)/n$ converges in probability, then
\begin{align}\label{CLT0}
\frac{S_{\tau_n}(\theta)}{\sqrt{I_{\tau_n}(\theta)}} &\rightarrow \cN(0,1) .
\end{align}
In order to do so, we  define the martingale-difference array
$$
Y_{nt}:=\frac{S_t(\theta)-S_{t-1}(\theta)}{\sqrt{n}}1_{\{t\leq \tau_n \}}, \quad t \in \bN, \quad n \in \bN.
$$
Indeed, since $\{S_t(\theta)\}$ is an $\{\cFt\}$-martingale and  $\tau_n$  an $\{\cFt\}$-stopping time, then  $\{t\leq \tau_n\}=\{\tau_n\leq t-1\}^{c}\in \cF_{t-1}$ and, consequently, we have
\begin{align*}
 \Exp_{\theta}[Y_{nt}|\cF_{t-1}]=\frac{1_{\{t\leq T_n\}}}{\sqrt{n}} \,  \Exp_{\theta} [S_t(\theta)-S_{t-1}(\theta) \, | \, \cF_{t-1}]=0.
\end{align*}
Moreover,  the increments of $\{S_{t}(\theta)\}$ are uniformly bounded by $K$, which implies that  for every $\epsilon>0$ we  have
\begin{align}\label{CLT2}
\sum_{t=1}^{\infty} \Exp_{\theta} [Y^2_{nt} \, \ind{|Y_{nt}|>\epsilon}]
\rightarrow   0
\end{align}
as $n \rightarrow \infty$.
 Therefore,    from the Martingale Central Limit Theorem (see, e.g. Theorem 35.12  in \cite{billi} and Slutsky's theorem it follows that if
\begin{align}\label{CLT1}
\sum_{t=1}^{\infty} \Exp[Y_{nt}^2 \, | \,\cF_{t-1}]= \frac{1}{n} \sum_{t=1}^{\tau_n} \Exp\left[ (S_t(\theta)-S_{t-1}(\theta))^2 \, | \, \cF_{t-1} \right]= \frac{I_{\tau_n}(\theta)}{n}
\end{align}
converges in probability to a positive number, then
\begin{align*}
\sqrt{\frac{n}{I_{\tau_n}(\theta)} } \sum_{t=1}^{\infty}Y_{nt} &=
\frac{1}{\sqrt{I_{\tau_n}(\theta)} }  \, \sum_{t=1}^{\tau_n}
\left[ S_t(\theta)-S_{t-1}(\theta) \right]  \\
&=    \frac{S_{\tau_n}(\theta)}{\sqrt{I_{\tau_n}(\theta)}} \longrightarrow \cN(0,1).
\end{align*}
If we now use  the  Taylor expansion \eqref{revi_taylor}, then the convergence \eqref{CLT0} takes the form
$$
\frac{I_{\tau_n}(\tilde{\theta}_{\tau_n})}{I_{\tau_n}(\theta)} \, \sqrt{I_{\tau_n}(\theta)} \, (\hat{\theta}_{\tau_n}-\theta) \rightarrow \cN(0,1),
$$
where $\tilde{\theta}_{\tau_n}$ lies between $\hat{\theta}_{\tau_n}$ and $\theta$. From the consistency of the estimator  \eqref{revi_consistent} it follows that  the ratio in the left-hand side goes to 1 almost surely and  from  Slutsky's theorem we obtain
$$
\sqrt{I_{\tau_n}(\theta)} \, (\hat{\theta}_{\tau_n}-\theta) \rightarrow \cN(0,1).
$$
From \eqref{revi_consistent} and another application of Slutsky's theorem we now obtain   \eqref{rev_asy_nor}. Finally, the second part  follows from the discussion that lead to Theorem \ref{CLTRCAT}.
\end{proof}


Therefore,  the proposed design leads to the same asymptotic behavior as that in the regular CAT design, as long as the proportion of revisions is small  relative to the number of distinct items. We expect that this is typically the case in practice, since most examinees tend to review and revise only a few items which they are not sure during the test process or at the end of the test.

\section{Simulation study}
We now present the results of a  simulation study that illustrates
the proposed design and our asymptotic results in a CAT with $n=50$ items. 
We consider items with  $m=3$ categories, thus,  each item can be revised at most once whenever revision is allowed. The parameters of the nominal response model are restricted in the following intervals $a_2 \in [-0.18, 4.15]$,  $a_3 \in [0.17, 3.93]$, $c_2 \in [  -8.27, 6.38]$ and  $c_3 \in [-7.00, 8.24]$, whereas we set $a_1=c_1=0$, which were selected based on a discrete item pool in Passos, Berger \& Frans E. Tan \cite{passos}. The analysis was  replicated  for  $\theta$ in $\{-3,-2,-1,0,1,2,3\}$.

With respect to the revision strategy,  we assume that the examinee decides to revise the $t^{th}$ question with probability, $p_t$. If we denote the total number of items which can be revised during the test as $n_1$, then $p_t$ satisfies the following recursion: $p_{t+1}=p_{t}-0.5/n_1$, $p_1=0.5$. For $n_1$ we considered the following possibilities:  $n_1/n=0.1,0.5,1$. Moreover, we assumed that whenever the examinee decides to revise,   each of the previous  items that have not been revised yet are equally likely to be selected.

For each of the above scenarios, we computed the  root mean square error (RMSE) of the final estimation  on the basis of $1000$ simulation runs. The  results are summarized in  Table \ref{tabler}. Note that when revision is allowed, the design is denoted as RCAT. We observe that revision often  improves  the ability estimation, especially  when the number of revisions is large. However, the RMSE is typically larger than the square root of the asymptotic variance, $\sqrt{n J^{*}(\theta)}$.  An exception seems to be the case that  $\theta=-2$ with a large number of  revisions.  In order to understand this further, we plot  in  Figure \ref{Fig1}  the evolution of the  total information $I_{t}(\theta)/ t$ (solid line with circles), the  information from the first responses, $\sum_{i=1}^{f_t} J(\theta,\bv_i) / f_t$ (dashed line with squares), the information from revisions $I^{R}_{t}(\theta) /t$ (dashed line with diamonds), where $1 \leq t \leq \tau_n$ and $I^{R}(\theta)$ is defined in \eqref{decompose}.   The horizontal line represents the asymptotic variance $J^{*}(\theta)$. Thus, we see that thanks to the contribution from a large number of revisions,  it is possible to outperform the best asymptotic performance that can be achieved in a standard CAT design.

Finally,   we  plot  in Figure \ref{Fig2} the  ``confidence intervals"   that would be obtained after $i$ items have been completed  in the case of a standard CAT, as well as when revision is allowed (in the case that $\theta=3$). Our asymptotic results suggests their validity for a large number of items and our graphs illustrate that revision seems  to actually improve the estimation of $\theta$.

\begin{table}
\caption{RMSE  in CAT and RCAT} 
\label{tabler}
\centering 
\begin{tabular}{ c c c c c c} 
\hline
  $\theta$ & $\sqrt{nJ^{*}(\theta)}$  &  CAT  &  & RCAT  &  \\
\hline
           &                                     &       &   & Expected number  & \\
           &                                     &       &   & of revisions     &  \\
\hline
           &                                 &        &  4      & 18       & 26 \\ [0.5ex]
\hline
  -3 & 0.0985 & 0.1042 & 0.1051 & 0.1068 & 0.1001 \\
  -2 & 0.0713 & 0.0746 & 0.0731 & 0.0701 & 0.0700 \\
  -1 & 0.0681 & 0.0716 & 0.0724 & 0.0718 & 0.0714 \\
  0  & 0.0681 & 0.0743 & 0.0723 & 0.0722 & 0.0721 \\
  1  & 0.0683 & 0.0773 & 0.0716 & 0.0699 & 0.0704 \\
  2  & 0.0681 & 0.0747 & 0.0718 & 0.0702 & 0.0701 \\
  3  & 0.0710 & 0.0787 & 0.0756 & 0.0728 & 0.0721 \\
\hline
\end{tabular}

\end{table}

\begin{figure}
\centering
  \includegraphics [height=60mm,width=120mm]{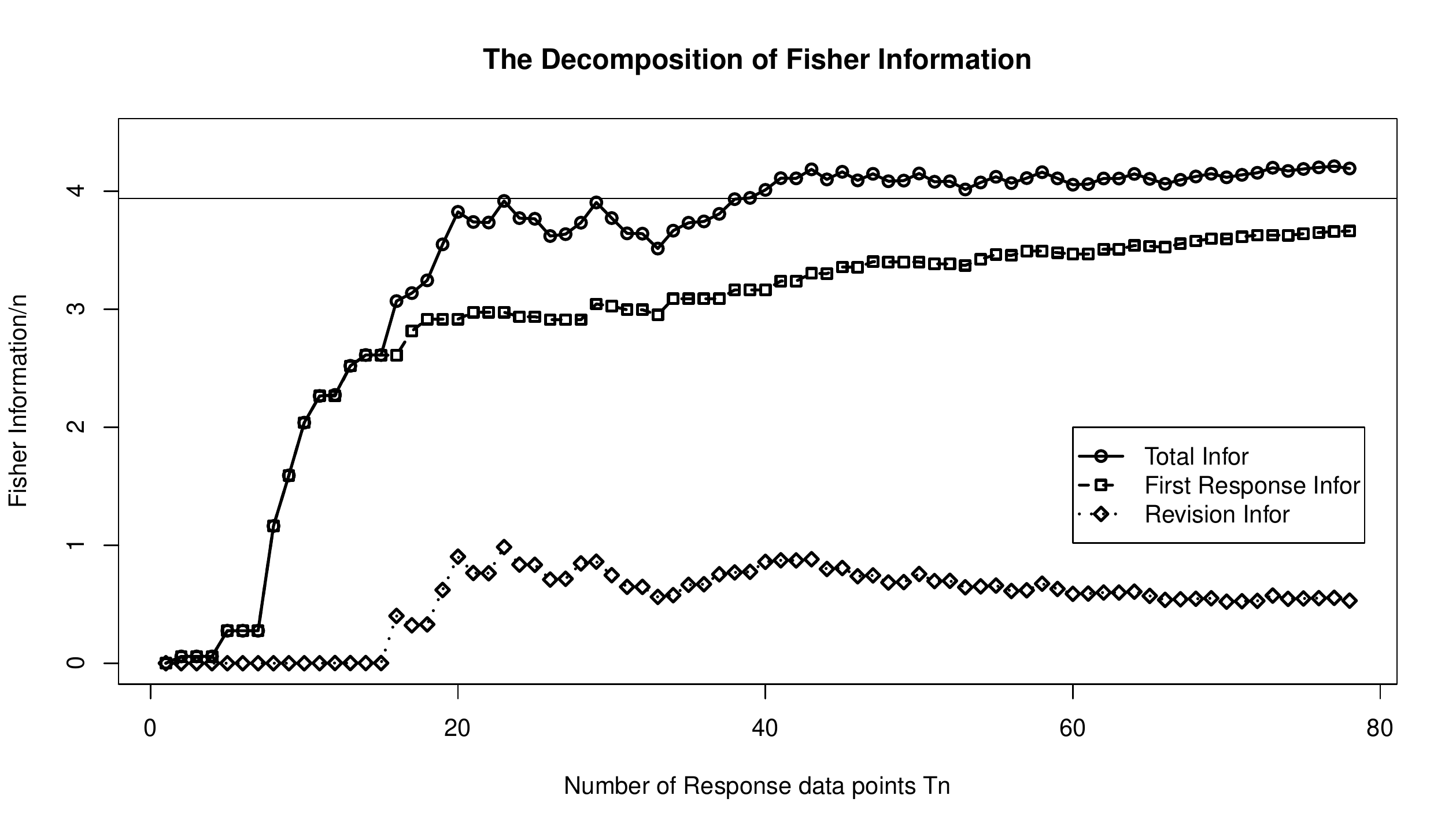}\\
  \caption{\small {The solid line represents the evolution of the normalized Fisher information,  that is $\{I_{t}(\hat{\theta}_{t})/ t, 1 \leq t \leq \tau_n\}$,  in a CAT with response revision.   The dashed line with squares   represents  the  information from the first responses and the dashed line with diamonds  the information from revisions, according to the decomposition  \eqref{decompose}. The true ability value is  $\theta=-2$.}}
\label{Fig1}
\end{figure}



\begin{figure}[]
\centering
  \includegraphics [height=60mm,width=120mm]{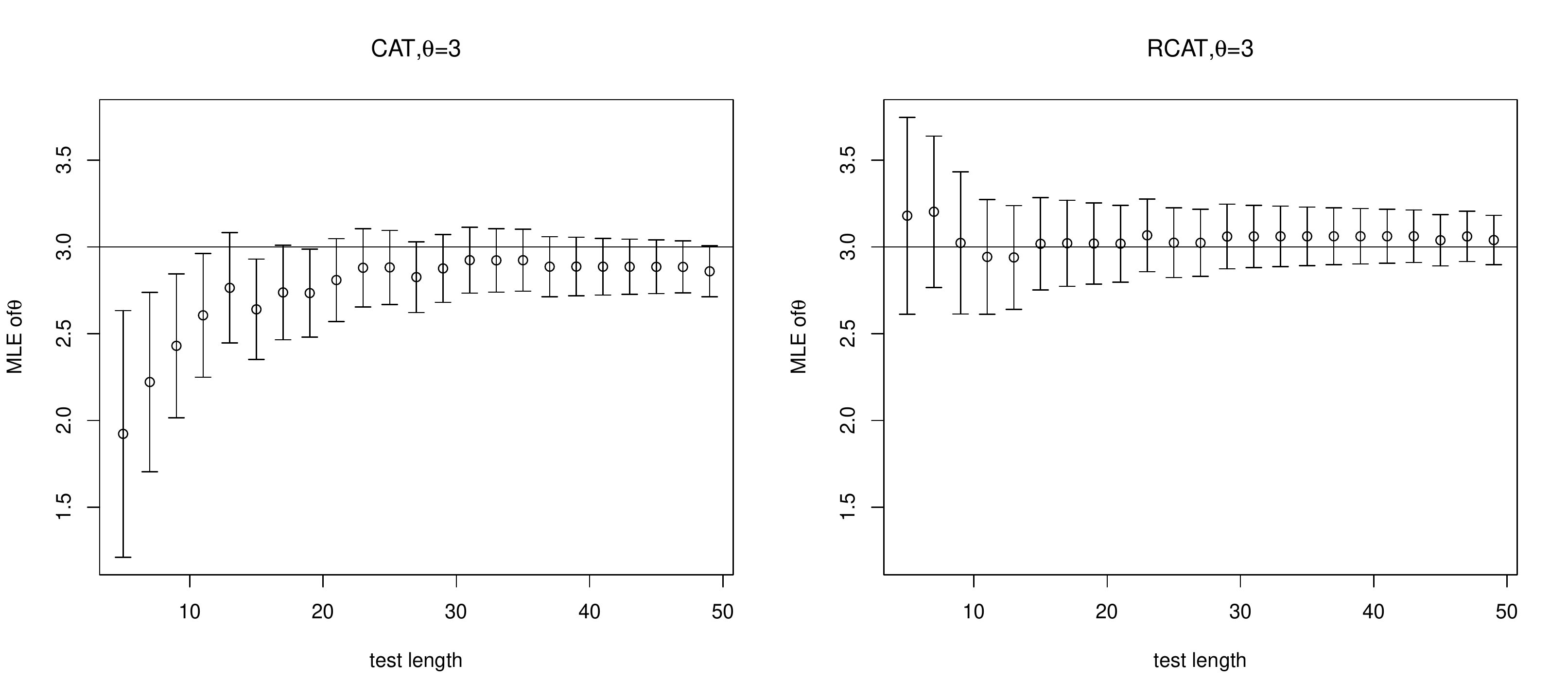}\\
  \caption{\small{The plot in the  left-hand side presents the  intervals $\hat{\theta}_{i} \pm 1.96 \cdot (I_{i}(\hat{\theta}_{i}))^{1/2}$, $1 \leq i \leq n$ in the case of the standard CAT. The plot in the right-hand side presents the  intervals $\hat{\theta}_{\tau_i} \pm 1.96 \cdot  (I_{\tau_i}(\hat{\theta}_{\tau_i}))^{1/2}$, $1 \leq i \leq n$ in the case where response revision is allowed. In both cases,  the true value of $\theta$ is $3$.}}
  \label{Fig2}
\end{figure}

\section{Conclusions}

In the first  part of this work,  we considered the   design of CAT that is based on the  nominal response model. Assuming conditional independence of the responses given the selected items and that the item parameters belong to a  bounded set, we established the strong consistency  of the MLE  for any item selection strategy and its asymptotic efficiency when the items are selected to maximize the current level of Fisher information. It is interesting to note that in the  special case of binary items ($m=2$) the  nominal response model reduces to the dichotomous 2PL model and, in this context, our  results complement the ones that were obtained in  \cite{r2} under the same model. Indeed, here  we assume that all item parameters belong to a bounded set, whereas in  \cite{r2}  it is assumed that the item difficulty parameter, $b$,  is unbounded,  a rather unrealistic assumption in practice where items are drawn from a given item bank. Moreover, we establish the strong consistency of the MLE \textit{for any item selection strategy}, unlike \cite{r2} where this is done only  when $b_i=\theta_{i-1}$.   Finally, from a technical point of view,  while the proofs  in  \cite{r2}  are heavily based on this closed-form expression for the $b_i$'s,  here we do not explicitly use this expression in our proofs (since it is not available in the general case of the nominal response model anyway).
  
  In the second part of this work, we  proposed a novel CAT design in which response revision is allowed. We showed that the proposed estimator is strongly consistent and that it becomes asymptotically normal (with the same asymptotic variance as in the standard CAT) when  items are selected to maximize the Fisher information at the current  ability estimate and the number of revision is small relative to the total number of items.
We further illustrated our theoretical results with a  simulation study.

From a policy point of view, our main message is that  the  nominal  response model  should be used for the design of CAT for two reasons. First,  because it  captures more  information than dichotomous models which collapse all possible wrong answers of an item to one  category. Second,   because it can be used in a natural way to allow for response revision. In fact,  one of the most appealing aspects of our approach is that it incorporates response  revision  \textit{without any additional calibration effort} than the one needed by the standard CAT that is based on the nominal response model.

Our work provides the first rigorous analysis  of  a  CAT design in which response revision is allowed and it opens a number of research directions. First of all,  items in reality are drawn without replacement from a finite pool.  This may call for modifications of the item selection strategy in order  to make  the proposed scheme more
robust (see, e.g.,  \cite{rr}). Moreover, more empirical work is required in order to understand the effect of  response revision on the ability  estimation, which can be much more substantial in practice than in the (idealistic) setup of our simulation study.

While our approach is robust, in the sense that we do not explicitly model the decision of the examinee to revise or not at each step given the selected items, it may result in a  loss of efficiency when the revision strategy  depends on the ability of the examinee. Modeling  this behavior is a challenge that could be addressed as soon as CATs that allow for response revision begin to be implemented in practice and relevant data can be obtained. Finally, it remains an open problem to  incorporate response revision in the case of  binary items, where a  dichotomous IRT model needs to be used and our approach cannot be applied.

\end{document}